\documentclass[11pt]{article}

\usepackage{fancyhdr}

\usepackage{setspace}
\usepackage[margin=1in]{geometry}
\usepackage{amsmath} 
\usepackage{amsthm}
\usepackage{amssymb}
\usepackage{hyperref} 
\usepackage{color}
\usepackage{graphicx}

\newtheorem{theorem}{Theorem}[section]
\newtheorem{lemma}[theorem]{Lemma}
\newtheorem{proposition}[theorem]{Proposition} 
\newtheorem{corollary}[theorem]{Corollary} 
\theoremstyle{definition}
\newtheorem{definition}[theorem]{Definition} 
\newtheorem{remark}[theorem]{Remark}

\fancypagestyle{plain}{% Redefine ``plain'' style for chapter boundaries  
\fancyhf{} % clear all header and footer fields  
\fancyfoot[R]{\thepage} % except the center  
}

\title{Generating Set for Nonzero Determinant Links Under Skein Relation}
\author{Aayush Karan}
\date{}

\begin{document}
\maketitle

\begin{abstract}
Traditionally introduced in terms of advanced topological constructions, many link invariants may also be defined in much simpler terms given their values on a few initial links and a recursive formula on a skein triangle. Then the crucial question to ask is how many initial values are necessary to completely determine such a link invariant. We focus on a specific class of invariants known as nonzero determinant link invariants, defined only for links which do not evaluate to zero on the link determinant. We restate our objective by considering a set $\mathcal{S}$ of links subject to the condition that if any three nonzero determinant links belong to a skein triangle, any two of these belonging to $\mathcal{S}$ implies that the third also belongs to $\mathcal{S}$. Then we aim to determine a minimal set of initial generators so that $\mathcal{S}$ is the set of all links with nonzero determinant. We show that only the unknot is required as a generator if the skein triangle is unoriented. For oriented skein triangles, we show that the unknot and Hopf link orientations form a set of generators.

\end{abstract}

\pagenumbering{arabic}
\section{Introduction}
Knots have long been appreciated for their aesthetic qualities, most notably in ancient Celtic art and design. However, a series of innovations in algebraic topology and combinatorics during the $20^{\text{th}}$ century revolutionized knots into an area of significant mathematical interest.

A knot is an embedding of the circle $\mathcal{S}^1$ into $\mathcal{S}^3$, the three-sphere. More generally, we refer to an intertwined collection of knots as a link, where these constituent knots are the \textit{components}. If no orientation on any component is specified, the link is \textit{unoriented}. However, if each component is assigned an orientation, the resulting link is \textit{oriented}. Links are often more effectively studied through their two-dimensional representations, known as \textit{link diagrams}, obtained by projecting the link onto a plane. The image of two overlapping strands under such a projection is a \textit{crossing} (Figure 1).

\begin{center}
\includegraphics[width = 1.5cm]{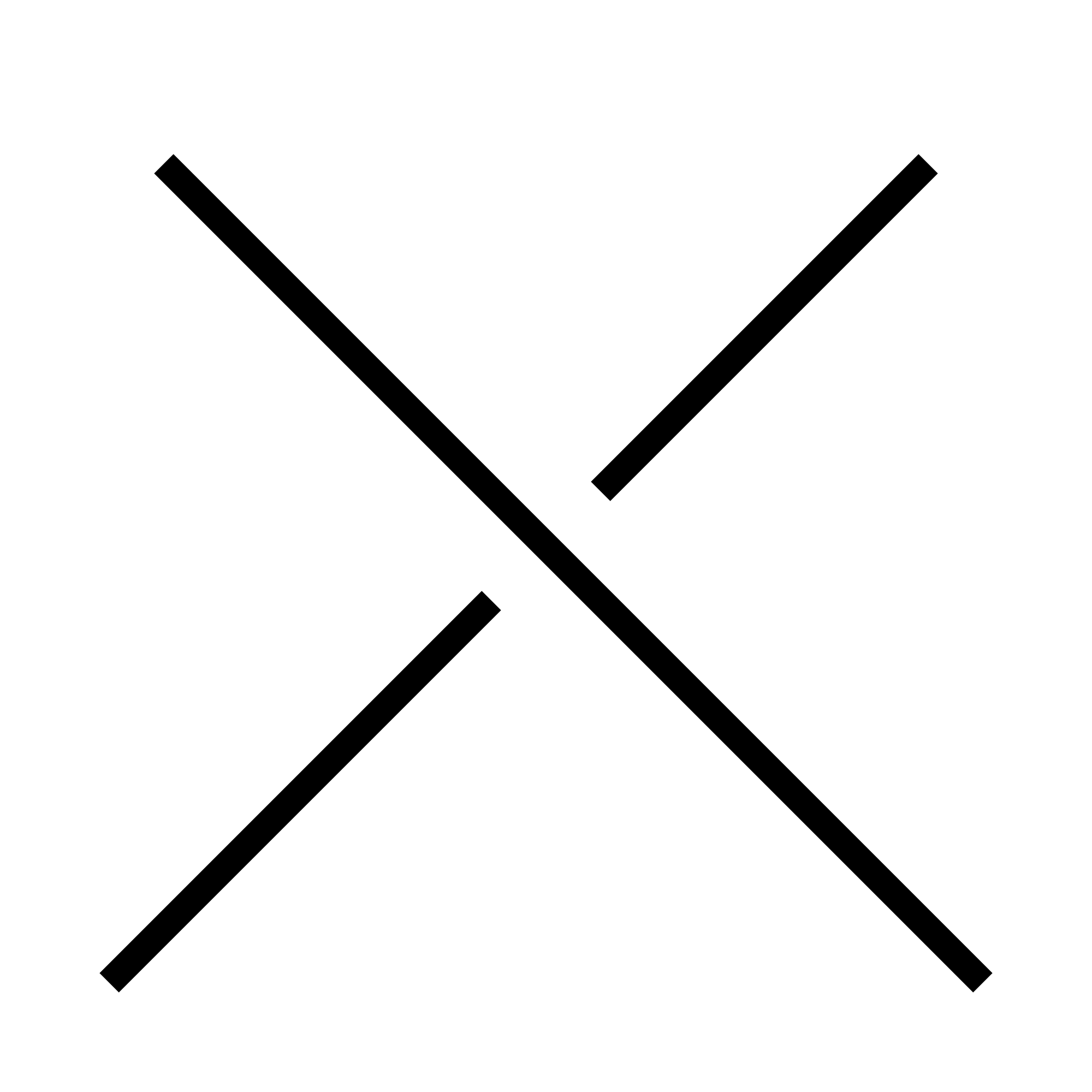}

Figure 1: Crossing
\end{center}

One of the driving objectives in knot theory has been determining whether two given links are equivalent, or \textit{isotopic}. Two links are isotopic if there exists a continuous deformation in $\mathcal{S}^3$ mapping one link to the other. In terms of link diagrams, Reidemeister showed that two links are isotopic if and only if their diagrams are related by a finite series of Reidemeister moves \cite{Lickorish}.

In practice, it is quite difficult to distinguish two links by directly proving no isotopy exists between them. Instead, it is much more feasible to show that the links have different \textit{link invariants}. A link invariant is any kind of mathematical object assigned to a link that remains unchanged under isotopy. This paper is mainly motivated by invariants defined solely for links with nonzero determinant, a quantity that will be defined shortly.

Discovered by J. W. Alexander in 1928, the Alexander polynomial is one of the earliest known link invariants, but it was originally defined in terms of advanced algebraic topological notions such as the knot group and homology \cite{Alexander}. However, with the advent of combinatorial methods in knot theory during the late $20^{\text{th}}$ century, mathematicians discovered that the Alexander polynomial, among other link invariants, could be more easily defined in terms of \textit{skein relations}. For a link invariant $\chi$, a skein relation relates the values of $\chi$ on a \textit{skein triangle}, which consists of three links whose diagrams are related through a geometric transformation near a single crossing. This paper focuses on two types of skein triangles: \textit{unoriented} and \textit{oriented}. 

The unoriented skein triangle, denoted $(L, L_0, L_1)$, is obtained by resolving a crossing from $L$ in two different ways (Figure 2a).
\begin{center}
\includegraphics[width = 7cm]{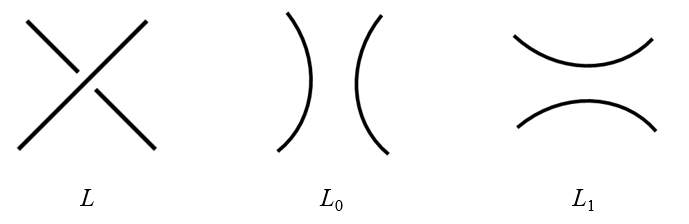}

Figure 2a: Unoriented skein triangle
\end{center}
The oriented skein triangle, denoted $(L_+, L_-, L_o)$, is a relation between oriented link diagrams obtained by a crossing change and the canonical orientation consistent resolution (Figure 2b).
\begin{center}
\includegraphics[width = 7cm]{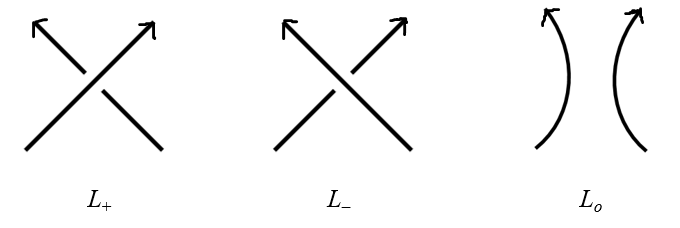}

Figure 2b: Oriented skein triangle
\end{center}
Examples of link invariants that satisfy skein triangles include the Alexander and Jones polynomials \cite{Lickorish}, the determinant, and the Casson-Walker invariant for branched double covers \cite{Mullins}.

If $\chi$ is defined for all links, we may recursively use the skein relation to completely determine $\chi$ given some initial values. The key question to ask is how many of these initial values we actually need. Without further restrictions, it is straightforward to see that only the values for the unknot (simple loop) and the unlinks (multiple unknots that are disjoint in the link diagram) are required, as any sequence of resolutions and crossing changes on a diagram will eventually result in a finite set of disjoint circles. The value of $\chi$ on two members of a skein triangle gives $\chi$ on the third, so we may repeatedly use the skein relation to obtain $\chi$ on all links. 

But some $\chi$ are more complex and only take values on certain types of links. In particular, we are concerned with \textit{nonzero determinant link invariants}, which, as the name suggests, are defined for links with nonzero \textit{determinant}. The determinant is obtained by evaluating the Alexander polynomial at $-1$ and taking the absolute value; however, we will examine a more elementary construction in Section $2$. Nonzero determinant link invariants, such as the Casson-Walker and Fr{\o}yshov invariants for branched double covers [6], are crucial to modern developments and applications of knot theory. Thus, our primary objective is to find a minimal set of initial values that completely determines such invariants on all nonzero determinant links by recursive application of a skein relation.

In $1993$, David Mullins proved the following theorem \cite{Mullins}: 

\begin{theorem}[Mullins]
Let $\chi$ be an invariant of non-zero determinant oriented links and suppose $\chi$ has
recursive formulas for the following triples of oriented links differing at a single crossing ($L_{u_i}$ is the non-canonical resolution of $L_+$ with appropriately chosen orientation): 

\[(L_+, L_-, L_o) \hspace{0.25cm} \text{whenever} 
\begin{array}{cc}
    \det(L_+) \\
     \det(L_-) \\ 
      \det(L_o) \\
\end{array} \Bigg\} \neq 0
\]

\[(L_+, L_-, L_{u_i}) \hspace{0.25cm} \text{whenever} 
\begin{array}{cc}
    \det(L_+) \\
     \det(L_-) \\ 
      \det(L_{u_i}) \\
\end{array} \Bigg\} \neq 0 \hspace{0.25cm} \text{and} \hspace{0.15cm} \det(L_o) = 0
\]

\[(L_+, L_o, L_{u_i}) \hspace{0.25cm} \text{whenever} 
\begin{array}{cc}
    \det(L_+) \\
     \det(L_o) \\ 
      \det(L_{u_i}) \\
\end{array} \Bigg\} \neq 0 \hspace{0.25cm} \text{and} \hspace{0.15cm} \det(L_-) = 0
\]

\[(L_-, L_o, L_{u_i}) \hspace{0.25cm} \text{whenever} 
\begin{array}{cc}
    \det(L_-) \\
     \det(L_o) \\ 
      \det(L_{u_i}) \\
\end{array} \Bigg\} \neq 0 \hspace{0.25cm} \text{and} \hspace{0.15cm} \det(L_+) = 0\]
Then $\chi$ is completely determined by these recursive formulas and its value on the
unknot. 
\end{theorem}

We will greatly simplify the conditions required to completely determine $\chi$ on both unoriented and oriented nonzero determinant links. To make this notion precise, we consider a subset $\mathcal{S}$ of all nonzero determinant links. 

\begin{definition}
$\mathcal{S}$ is \textit{closed under the unoriented skein triangle} if it is subject to the following condition: suppose three links each with nonzero determinant constitute an unoriented skein triangle. Then if any two of them belong to the set, so does the third.
One can define the notion of being \textit{closed under the oriented skein triangle} in a similar manner.
\end{definition}

We shall prove the following main theorems of this paper:

\begin{theorem}\label{main theorem unoriented}
Given a set $\mathcal{S}$ of unoriented nonzero determinant links, suppose that it is closed under the unoriented skein triangle and contains the unknot. Then it contains all nonzero determinant links. 
\end{theorem}

\begin{remark}(See \cite{Osvath} for more details.)
One can define a weaker version of being closed under the unoriented skein triangle by requiring that the two links belonging to the set $\mathcal{S}$ have determinant smaller than the third link in the skein triangle. This forms the set of quasi-alternating links defined by Ozsv\'ath-Szab\'o, an important family of links in modern knot theory. In particular, it is known  that the set of quasi-alternating links does not contain all links with nonzero determinant. Therefore, Theorem  \ref{main theorem unoriented} implies that Ozsv\'ath-Szab\'o's definition of quasi-alternating links is very delicate.
\end{remark}

Furthermore, if $\mathcal{S}$ is closed under the oriented skein relation, we must include the orientations of the Hopf link among our initial members:

\begin{theorem}\label{thm: oriented theorem}
Given a set $\mathcal{S}$ of oriented nonzero determinant links, suppose that it is closed under the oriented skein triangle and contains the unknot and all orientations of the Hopf link. Then it contains all nonzero determinant links.
\end{theorem}

These theorems imply that for any link $L$ with nonzero determinant, there exists a sequence of nonzero determinant links $(L^0, L^1, ..., L^n)$ with $L^0, L^1 \in \mathcal{S}$ and $L^n = L$ such that for any $L^m$ with $2 \leq m \leq n$, there exist $L^i$ and $L^j$ satisfying $i, j < m$ forming a skein triangle with $L^m$. Thus, for a nonzero determinant link invariant $\chi$, if its values on $L^0$ and $L^1$ are known, it is completely determined:

\begin{corollary}
Let $\chi$ be an invariant of non-zero determinant unoriented links and suppose $\chi$ has a recursive formula for the following triple of links differing at a single crossing:
\[(L, L_0, L_1) \hspace{0.25cm} \text{whenever} 
\begin{array}{cc}
    \det(L) \\
     \det(L_0) \\ 
      \det(L_1) \\
\end{array} \Bigg\} \neq 0
\]
Then $\chi$ is completely determined by its value on the unknot.
\end{corollary}

\begin{corollary}
Let $\chi$ be an invariant of non-zero determinant oriented links and suppose $\chi$ has a recursive formula for the following triple of links differing at a single crossing:
\[(L_+, L_-, L_o) \hspace{0.25cm} \text{whenever} 
\begin{array}{cc}
    \det(L_+) \\
     \det(L_-) \\ 
      \det(L_o) \\
\end{array} \Bigg\} \neq 0
\]
Then $\chi$ is completely determined by its values on the unknot and all orientations of the Hopf link.
\end{corollary}

We establish basic properties of the link determinant and rational tangles (essential tool used in our proof) in Section $2$. In Section $3$, we begin by introducing key lemmas that relate the determinant to rational tangles and the unoriented skein triangle, subsequently proving Theorem $1$. In Section $4$, we conclude by proving Theorem $2$ using oriented modifications of the lemmas used in previous sections.

\section{Preliminaries}\label{sec:preliminaries}
\subsection{The Link Determinant}

We begin by examining the link determinant and its properties. Given a link $L$ and diagram with $k$ crossings, label the arcs with formal variables (or \textit{colors}) $c_1, c_2, ..., c_k$ in some order. Consider a crossing as shown in figure $3$. For a fixed prime $n$, we associate the colors with residues $(\bmod\; n)$ such that 
\[2c_p - c_q - c_r \equiv 0 \pmod{n}\]
\begin{center}
    \includegraphics[width = 2.5cm]{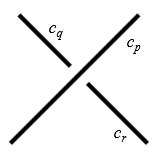}
    
    Figure 3: Labeled crossing
\end{center}
over all crossings. If there exists a non-monochromatic solution in $c_1, c_2, ..., c_k$, we say that $L$ is \textit{$n$-colorable}. It is straightforward to check that coloring is invariant under the three Reidemeister moves, so we can indeed refer to the colorability of $L$. Our system of $k$ equations in $k$ variables can be represented as a matrix equation with coefficient matrix $\mathcal{C}$, the \textit{coloring matrix} of the diagram.

Note that $\mathcal{C}$ has determinant $0$; by definition, each row has exactly one occurrence of $2$ and two occurrences of $-1$, so the column vectors sum to $0$. However, the coloring matrix has the interesting property that all minors of size $(k-1) \times (k-1)$ have the same determinant (see \cite{Kauffman}), motivating the definition of the link determinant:

\begin{definition}
The \textit{determinant} of link $L$, denoted $\det(L)$, is the absolute value of the determinant of the matrix obtained by removing a row and column from $\mathcal{C}$. As a convention, the unknot has determinant $1$.
\end{definition} 

This definition implicitly claims the determinant is a link invariant, the proof of which we will refer the reader to \cite{Kauffman}. With $\det(L)$, we are able to completely determine the colorability of $L$:

\begin{lemma}\label{lem: p-colorable}
 $L$ is $n$-colorable for $n$ prime if and only if $n$ divides $\det(L)$.
\end{lemma}

\begin{proof}
Work in $\mathbb{F}_n$, and suppose $L$ is $n$-colorable. Note that vectors of the form $a \cdot \overline{1}$ for all scalars $a$ are in the nullspace of $\mathcal{C}$, where $\overline{1}$ is the vector with all $1$'s. If  the nullity of $\mathcal{C}$ (denoted null($\mathcal{C}$)) is $1$, then these are the only vectors in the nullspace, corresponding to monochromatic colorings. Since we want nontrivial colorings, we must have null$(\mathcal{C}) \geq 2$, whence rank$(\mathcal{C}) = k - $null$(\mathcal{C}) \leq k - 2$. If there exists a minor of size $k-1$  with nonzero determinant, we obtain $k-1$ linearly independent rows, contradicting rank$(\mathcal{C}) \leq k - 2$. This forces $\det(L) \equiv 0 \pmod{n}$ if $L$ is $n$-colorable. Similarly in the reverse direction, $\det(L) \equiv 0 \pmod{n}$ implies all minors of size $(k-1) \times (k-1)$ have zero determinant, so rank$(\mathcal{C}) \leq k - 2$. Hence $L$ is $n$-colorable if and only if $n$ divides $\det(L)$. 
\end{proof}

In turn, we can use this lemma to help characterize the determinant:

\begin{lemma}\label{lem: knot odd determinant}
A link has odd determinant if and only if it is a knot. In particular, knots have nonzero determinant.
\end{lemma}

\begin{proof}

If the link has more than one component, we may simply color entire components with residues $0$ or $1$ $(\bmod\; 2)$ with at least one of each color. Therefore, the link is 2-colorable, and by Lemma \ref{lem: p-colorable}, we conclude that it has even determinant. 

On the other hand, suppose a knot is $2$-colorable. Note that the two broken arcs adjacent in a crossing must be the same color. Travel along an arc, passing to the adjacent broken arc every time a crossing is reached. Each arc is the same color as the previous one, and since there is only one component, the path hits each arc in the diagram. However, this implies a monochromatic coloring, contradiction. Thus a knot is never $2$-colorable, so it always has odd determinant.
\end{proof}

The link determinant also behaves nicely with combining links. A trivial combination of links is known as a \textit{split configuration}:

\begin{definition}
A \textit{split configuration} is a link such that two of its component links may be separated to disjoint $3$-balls via an isotopy.
\end{definition}

\begin{lemma}\label{lem: split config}
A split configuration $L \cup L'$  has zero determinant. 
\end{lemma} 

\begin{proof}
Let $L$ and $L'$ have coloring matrices $\mathcal{C}$ and $\mathcal{C}'$ respectively. Then it is straightforward to check that the coloring matrix for $L \cup L'$ is the block matrix

\[
M=
  \begin{bmatrix}
    \mathcal{C} & 0 \\
    0 & \mathcal{C}'
  \end{bmatrix}.
\]

We can remove the first row and column from $M$ to obtain $M'$ with reduced first block $\mathcal{D}$, so $\det(L \cup L') = \det(M') = \det(\mathcal{D}) \cdot \det(\mathcal{C}') =  \det(L) \cdot 0 = 0$. 
\end{proof}

The other combination of interest is the \textit{connected sum} of $L_1$ and $L_2$, denoted by $L_1 \# L_2$, which is obtained by breaking each link apart at a strand and connecting the loose ends to each other. Then the determinant acts multiplicatively:

\begin{proposition}\label{lem: connected sum multiply}
(Proposition $6.12$ in \cite{Lickorish}): For any two links $L_1$ and $L_2$, $\det(L_1 \# L_2) = \det(L_1) \cdot  \det(L_2)$.
\end{proposition} 
\begin{flushleft}
These above properties help us better characterize links with zero determinant.
\end{flushleft}
\subsection{Rational Tangles}\label{subsec: rational tangles}
A $n$-tangle is a proper embedding of a disjoint union of $n$ arcs in the $3$-ball. We specifically focus on a family of $2$-tangles known as \textit{rational tangles}. Rational tangles provide a method of assigning rational numbers to certain sub-configurations of links. These configurations act as building blocks for links, giving us a powerful algebraic perspective highly compatible with the determinant. Following the construction from \cite{Mellor and Nevin}, a rational tangle is any tangle obtained by repeatedly twisting adjacent endpoints of the trivial tangles (two parallel horizontal or vertical strands) to create crossings. In particular, we can define the \textit{elementary rational tangles}, denoted by $[n]$ and $\frac{1}{[n]}$ for $n \in \mathbb{Z}$, referring to $|n|$ half-twists (right-handed twist if $n$ is positive and left-handed if $n$ is negative) performed on the horizontal and vertical trivial tangles respectively (Figure 4). We can also define the tangle operations $+$ and $*$ (Figure 5), which respectively merge tangles horizontally and vertically. Note that $+$ and $*$ are commutative up to isotopy, so these operations allow us to describe any rational tangle as the result of iteratively adjoining the elementary tangles to each other.

Indeed, let $T_0$ denote the choice of initial trivial tangle (either $[0]$ or $\frac{1}{[0]}$), and consider the following steps:

\begin{enumerate}
\item $T_i = T_{i-1} + [a_i]$ for some nonzero $a_i \in \mathbb{Z}$.

\item $T_i = T_{i-1} * \frac{1}{[a_i]}$ for some nonzero $a_i \in \mathbb{Z}$.
\end{enumerate}
\begin{center}
    \includegraphics[width = 12cm]{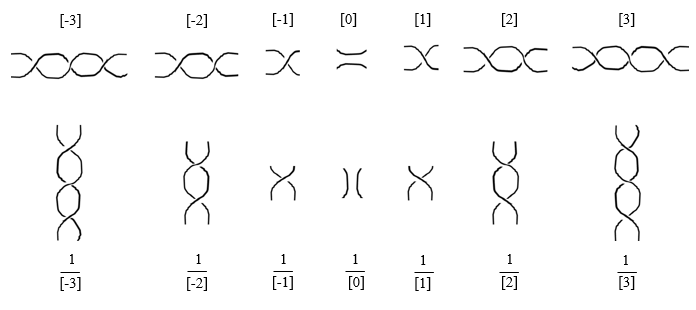}
    
    Figure 4: Rational tangles $[n]$ and $\frac{1}{[n]}$ for $-3 \leq n \leq 3$.
\end{center}

\begin{center}
    \includegraphics[width = 12cm]{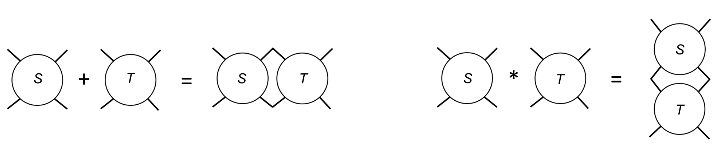}
    
    Figure 5: Tangle operations
\end{center}

Then we alternate between the two steps, beginning with the former if $T_0 = [0]$ and  the latter if $T_0 = \frac{1}{[0]}$. For a choice of nonzero integers $(a_1, a_2, ..., a_n)$ we either get the tangle

\[\bigg(\bigg([a_1]*\frac{1}{[a_2]}\bigg) + [a_3] \bigg) * ... \hspace{0.2cm} \text{or} \hspace{0.2cm} \bigg(\bigg(\frac{1}{[a_1]} + [a_2]\bigg) * \frac{1}{[a_3]} \bigg) + ... \]

But since $[0] + T = T = \frac{1}{[0]}*T$ for any tangle $T$, we can combine the above expressions by allowing $a_1$ to equal $\frac{1}{0}$ and $a_n$ to equal $0$, yielding

\[\bigg(...\bigg(\bigg([a_1]*\frac{1}{[a_2]}\bigg) + [a_3] \bigg)* ... * \frac{1}{[a_{n-1}]} \bigg) + [a_n]\]
for an appropriate choice of $a_1, ..., a_n$ and odd positive integer $n$. We represent this resulting rational tangle by its \textit{fraction}:

\[ (a_1, a_2, ..., a_n) = a_n + \frac{1}{a_{n-1} + \frac{1}{... + \frac{1}{a_2 + \frac{1}{a_1}}}}.\]

As one would hope for, the fraction uniquely determines a rational tangle, as claimed in \cite{Mellor and Nevin}:

\begin{proposition} Two rational tangles are isotopic if and only if they have the same fraction.
\end{proposition} 

We can now identify any rational tangle with a fraction $\frac{p}{q} \in \mathbb{Q}^+ = \mathbb{Q} \cup \left\{\frac{1}{0}\right\}$. Rational tangles in fact appear as sub-configurations in a link, which we distinguish by drawing a disk around the tangle in the link diagram representation, intersecting the tangle at $4$ endpoints. There is a useful relationship between the parities of the numerator and denominator of a rational tangle and the way the tangle connects the endpoints of the disk. This relationship helps us better characterize how the various strands of a link are connected.

\begin{lemma}\label{lem: endpoint connection}
Let endpoints be labeled $a, b, c,$ and $d$ as in figure 6, and consider a rational tangle $\frac{p}{q}$ to be inserted into the disk. Then the following statements are true:
\begin{itemize}
\item If $p$ is even and $q$ is odd, $a$ is connected to $b$ and $c$ is connected to $d$.
\item If $p$ is odd and $q$ is even, $a$ is connected to $c$ and $b$ is connected to $d$.
\item If $p$ is odd and $q$ is odd, $a$ is connected to $d$ and $b$ is connected to $c$. 
\end{itemize}
\end{lemma}

\begin{center}
\includegraphics[width = 3cm]{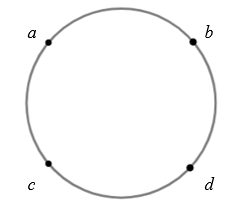}

Figure 6: Disk with endpoints
\end{center}

\begin{proof}
A rational tangle $\frac{p}{q}$ can be represented by the continued fraction
\[a_n + \frac{1}{a_{n-1} + \frac{1}{... + \frac{1}{a_2 + \frac{1}{a_1}}}} = \frac{p}{q},\]
where $n$ is odd, $a_1 \in \mathbb{Z}\setminus{\left\{0\right\}} \cup \left\{\frac{1}{0}\right\}$, $a_2, ..., a_{n-1} \in \mathbb{Z}\setminus{\left\{0\right\}}$, and $a_n \in \mathbb{Z}$. We define a sequence of \textit{truncated continued fractions} $\big(\frac{p_k}{q_k}\big)_{1 \leq k \leq n}$ by the expressions
\[\frac{p_{2i-1}}{q_{2i-1}} = a_{2i-1} + \frac{1}{a_{2i-2} + \frac{1}{... + \frac{1}{a_2 + \frac{1}{a_1}}}} \hspace{0.25cm} \text{and} \hspace{0.25cm} \frac{p_{2j}}{q_{2j}} = \frac{1}{a_{2j} + \frac{1}{a_{2j-1} + \frac{1}{... + \frac{1}{a_2 + \frac{1}{a_1}}}}}\]
for $1 \leq i \leq \frac{n+1}{2}$ and $1 \leq j \leq \frac{n-1}{2}$. If $\mathbb{Q}_*^+$ is the set of reduced fractions in $\mathbb{Q}^+$, consider the map $\psi:\mathbb{Q}_*^+ \longrightarrow \mathbb{Z} \times \mathbb{Z}$ sending the fraction $\frac{a}{b}$ to the vector $\binom{a}{b}$. Observe that this vector representation yields a recursive relation between the truncated continued fractions, namely
\[\left(
\begin{array}{c}
p_{2i}\\
q_{2i}\\
\end{array}
\right) = \left( {\begin{array}{cc}
   1 & 0 \\
   a_{2i} & 1 \\
  \end{array} } \right) \left(
\begin{array}{c}
p_{2i-1}\\
q_{2i-1}\\
\end{array}
\right) \hspace{0.25cm} \text{and} \hspace{0.25cm} \left(
\begin{array}{c}
p_{2i+1}\\
q_{2i+1}\\
\end{array}
\right) = \left( {\begin{array}{cc}
   1 & a_{2i+1} \\
   0 & 1 \\
  \end{array} } \right) \left(
\begin{array}{c}
p_{2i}\\
q_{2i}\\
\end{array}
\right).
\]
Thus, we may write
\[\left(
\begin{array}{c}
p\\
q\\
\end{array}
\right) = \left( {\begin{array}{cc}
   1 & a_n \\
   0 & 1 \\
  \end{array} } \right) \left( {\begin{array}{cc}
   1 & 0 \\
   a_{n-1} & 1 \\
  \end{array} } \right)...\left( {\begin{array}{cc}
   1 & 0 \\
   a_2 & 1 \\
  \end{array} } \right) \left(
\begin{array}{c}
p_1\\
q_1\\
\end{array}
\right).\]

Recall the construction for the rational tangle $\frac{p}{q}$ is given by 
\[\bigg(...\bigg(\bigg([a_1]*\frac{1}{[a_2]}\bigg) + [a_3] \bigg)* ... * \frac{1}{[a_{n-1}]} \bigg) + [a_n].\]
To each rational tangle $T$ we assign a \textit{connectivity vector} $\vec{c}_T \in \left\{\binom{0}{1}, \binom{1}{0}, \binom{1}{1}\right\}$ such that $\vec{c}_T = \binom{0}{1}$ if $T$ connects $a$ to $b$ and $c$ to $d$, $\vec{c}_T = \binom{1}{0}$ if $T$ connects $a$ to $c$ and $b$ to $d$, and $\vec{c}_T = \binom{1}{1}$ if $T$ connects $a$ to $d$ and $b$ to $c$. Working in $\mathbb{F}_2$, we can relate $\vec{c}_T$ to the connectivity vector of the tangle obtained by adjoining a elementary rational tangle to $T$ using
\[\vec{c}_{T*\frac{1}{[k]}} = \left( {\begin{array}{cc}
   1 & 0 \\
   k & 1 \\
  \end{array} } \right)\vec{c}_T \hspace{0.25cm} \text{and} \hspace{0.25cm} \vec{c}_{T+\frac{[k]}{1}} = \left( {\begin{array}{cc}
   1 & k \\
   0 & 1 \\
  \end{array} } \right)\vec{c}_T.\]
The first expression holds because adjoining $\frac{1}{[k]}$ through the operation $*$ either preserves the existing connectivity or swaps whether $a$ is connected to $c$ or $d$. When $k$ is even, the connectivity is preserved, so $\left( {\begin{array}{cc}
   1 & 0 \\
   k & 1 \\
  \end{array} } \right)$, which is equivalent to $\left( {\begin{array}{cc}
   1 & 0 \\
   0 & 1 \\
  \end{array} } \right)$ in $\mathbb{F}_2$, acts as the identity. When $k$ is odd, adjoining $\frac{1}{[k]}$ swaps $c$ and $d$. If $\vec{c}_T = \binom{0}{1}$, $T$ connects $a$ to $b$ and $c$ to $d$, so the swap does not affect this connectivity. Otherwise, T connects $a$ to $c$ or $d$, which becomes connected to $d$ or $c$ respectively after the swap. These transformations are reflected by the following computations:
  \[\left( {\begin{array}{cc}
   1 & 0 \\
   1 & 1 \\
  \end{array} } \right) \left(
\begin{array}{c}
0\\
1\\
\end{array}
\right) = \left(
\begin{array}{c}
0\\
1\\
\end{array}
\right), \hspace{0.5cm} \left( {\begin{array}{cc}
   1 & 0 \\
   1 & 1 \\
  \end{array} } \right) \left(
\begin{array}{c}
1\\
0\\
\end{array}
\right) = \left(
\begin{array}{c}
1\\
1\\
\end{array}
\right), \hspace{0.25cm} \text{and} \hspace{0.25cm} \left( {\begin{array}{cc}
   1 & 0 \\
   1 & 1 \\
  \end{array} } \right) \left(
\begin{array}{c}
1\\
1\\
\end{array}
\right) = \left(
\begin{array}{c}
1\\
0\\
\end{array}
\right).\]
The expression relating $\vec{c}_{T+\frac{[k]}{1}}$ to $\vec{c}_T$ follows from a similar argument, so we may write
\[\vec{c}_{\, \frac{p}{q}} = \left( {\begin{array}{cc}
   1 & a_n \\
   0 & 1 \\
  \end{array} } \right) \left( {\begin{array}{cc}
   1 & 0 \\
   a_{n-1} & 1 \\
  \end{array} } \right)...\left( {\begin{array}{cc}
   1 & 0 \\
   a_2 & 1 \\
  \end{array} } \right)\vec{c}_{[a_1]}\]
in $\mathbb{F}_2$. Since $[a_1]$ is an elementary rational tangle, $\vec{c}_{[a_1]} = \binom{p_1}{q_1}$, whence
\[\vec{c}_{\, \frac{p}{q}} = \left( {\begin{array}{cc}
   1 & a_n \\
   0 & 1 \\
  \end{array} } \right) \left( {\begin{array}{cc}
   1 & 0 \\
   a_{n-1} & 1 \\
  \end{array} } \right)...\left( {\begin{array}{cc}
   1 & 0 \\
   a_2 & 1 \\
  \end{array} } \right) \left(
\begin{array}{c}
p_1\\
q_1\\
\end{array}
\right) = \left(
\begin{array}{c}
p\\
q\\
\end{array}
\right),\]
as desired.
\end{proof}

Any given link $L$ can be trivially decomposed as a network of connected rational tangles; indeed, each crossing of its diagram can be considered as either the $\frac{1}{1}$ or $\frac{1}{-1}$ tangle. Thus we can view links in terms of the disks around the rational tangles, bringing us to the following definition:

\begin{definition}
A \textit{decomposition} of a link diagram is a representation of the diagram as a nexus of disks, each containing a rational tangle, such that every endpoint of a disk is connected to another and any crossing is contained within a disk (Figure 7).

\begin{center}
    \includegraphics[width = 6cm]{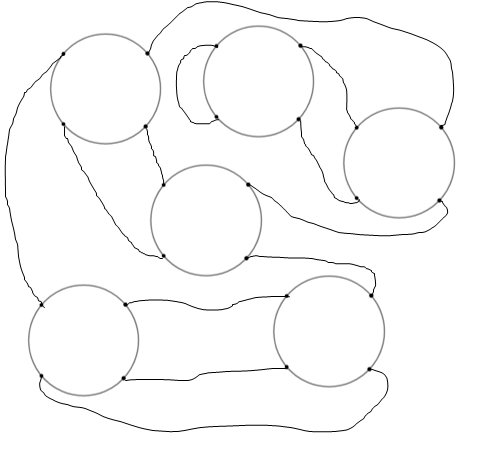}
    
    Figure 7: Decomposition of a link diagram with six disks, each containing a rational tangle.
\end{center}

\end{definition}
\begin{flushleft}
Of course we have the trivial decomposition, but we may obtain a more substantial notion:
\end{flushleft}
\begin{definition}
The \textit{complexity} of a link diagram is the minimal number of disks needed in its decomposition. The \textit{complexity} of a link is the minimal complexity among all its link diagrams.
\end{definition}

\begin{flushleft}
For the rest of the paper, we shall view links through the lens of the rational tangles they contain.
\end{flushleft}

\subsection{Skein Triangles and Rational Tangles}

Using the previous framework, we now provide an algebraic version of the unoriented and oriented skein triangle. By virtue of how rational tangles are constructed from elementary half-twists, they are extremely compatible with the skein triangles. Resolutions simplify the number of half-twists in an elementary tangle, giving us the following elegant relation:

\begin{lemma}\label{lem: unoriented skein}
If $\frac{a}{b}$ and $\frac{c}{d}$ contained in $\mathbb{Q}^+$ satisfy $|ad - bc| = 1$, then the rational tangles $\frac{a}{b}$, $\frac{c}{d}$, and $\frac{a+c}{b+d}$ form an unoriented skein triangle such that $\frac{a}{b}$ and $\frac{c}{d}$ are the resolutions of $\frac{a+c}{b+d}$.
\end{lemma}

\begin{proof}
Now suppose $a, b, c, $ and $d$ are all nonzero. Note $\frac{a}{b}$ and $\frac{c}{d}$ must have the same sign, otherwise $|ad - bc| = |ad| + |bc| \geq 2$. Because the rational tangle $-r$ is simply the mirror configuration of $r$, we may assume without loss of generality that $\frac{a}{b}$ and $\frac{c}{d}$ are positive. The central claim is that $\frac{a}{b}$ and $\frac{c}{d}$ can be represented as continued fractions $(s_1, s_2, ..., s_n)$ and $(t_1, t_2, ..., t_n)$ respectively such that $s_i = t_i$ for all $2 \leq i \leq n$ and $\left\{s_1, t_1\right\} = \left\{\frac{1}{0}, k\right\}$, where $k \in \mathbb{Z}^+$ and $n$ is not necessarily odd. 

If $\lfloor \frac{a}{b} \rfloor \neq \lfloor \frac{c}{d} \rfloor$, let $u \leq \frac{a}{b} < u + 1$ and $v \leq \frac{c}{d} < v + 1.$ Then $a = bu + i$ and $c = dv + j$ for $i \in [0, b-1]$ and $j \in [0, d-1]$. Without loss of generality suppose $u > v$, and observe that 
\[ad - bc = (bu + i)d - b(dv + j) = bd(u - v) + di - bj.\]
Since $0 \leq j < d$, 
\[bd(u - v) + di - bj > bd(u - v - 1) + di \geq 0.\]
If $bd(u - v - 1) + di \geq 1$, $ad - bc > 1$, which is a contradiction. Hence $bd(u - v - 1) + di = 0$, and because $b$ and $d$ are positive integers, we must have $u = v + 1$ and $i = 0$. Then $(a, b) = (v+1, 1)$, and plugging these values into our expression for $ad - bc$ yields
\[ad - bc = d(v+1) - (dv + j) = d - j = 1.\]
This forces $(c, d)$ to equal $(dv + d-1, d)$, from which  $\frac{a}{b}$ and $\frac{c}{d}$ must be of the form
\[\frac{a}{b} = \frac{v+1}{1}, \hspace{0.8cm} \frac{c}{d} = \frac{vd + d - 1}{d}\]
for some $v \geq 0$ and $d \geq 1$.

Thus, we can express $\frac{a}{b}$ and $\frac{c}{d}$ in the desired continued fractional form as
\[\frac{a}{b} = v + \frac{1}{1 + \frac{1}{\frac{1}{0}}}, \hspace{0.8cm} \frac{c}{d} = v + \frac{1}{1 + \frac{1}{d-1}}\]
if $d \geq 1$ and
\[\frac{a}{b} = v + \frac{1}{1}, \hspace{0.8cm} \frac{c}{d} = v + \frac{1}{\frac{1}{0}}\]
if $d = 0$.

Otherwise, $\lfloor \frac{a}{b} \rfloor = \lfloor \frac{c}{d} \rfloor = m$. Beginning the continued fractional representation of $\frac{a}{b}$ and $\frac{c}{d}$, we have $\frac{a}{b} = m + \frac{1}{x_1}$ and $\frac{c}{d} = m + \frac{1}{y_1}$, defining $x_1$ and $y_1$ to be $\frac{1}{\lbrace{\frac{a}{b}}\rbrace}$ and $\frac{1}{\lbrace{\frac{c}{d}}\rbrace}$ respectively (here $\lbrace{x\rbrace} = x - \lfloor x \rfloor$). We can continue these representations by iteratively writing (for $k \geq 1$) $x_k = \lfloor x_k \rfloor + \frac{1}{x_{k+1}}$ and $y_k = \lfloor y_k \rfloor + \frac{1}{y_{k+1}}$, where $x_{k+1} = \frac{1}{\lbrace{x_k\rbrace}}$ and $y_{k+1} = \frac{1}{\lbrace{y_k\rbrace}}$. Since $\frac{a}{b} \neq \frac{c}{d}$, there exists $N$ for which $\lfloor x_N \rfloor \neq \lfloor y_N \rfloor$ but $\lfloor x_t \rfloor = \lfloor y_t \rfloor$ for all $t < N$. The property $|ad - bc| = 1$ is preserved under subtracting an integer constant $M$ from both $\frac{a}{b}$ and $\frac{c}{d}$, as $\frac{a - bM}{b}$ and $\frac{c - dM}{d}$ are in lowest terms if $\frac{a}{b}$ and $\frac{c}{d}$ are and $|(a-bM)d - b(c - dM)| = |ad - bc| = 1$. In addition, this property is preserved under taking reciprocals, so if rationals $\lfloor x_N \rfloor$ and $\lfloor y_N \rfloor$ can respectively be written as $\frac{p}{q}$ and $\frac{r}{s}$ in lowest terms, $|ps - qr| = 1$. This allows us to use the representation described above for $\lfloor x_N \rfloor$ and $\lfloor y_N \rfloor$, yielding the desired continued fractional representation for $\frac{a}{b}$ and $\frac{c}{d}$. Thus, any positive rationals $\frac{a}{b}$ and $\frac{c}{d}$ satisfying $|ad - bc| = 1$ can be represented as continued fractions $(s_1, s_2, ..., s_n)$ and $(t_1, t_2, ..., t_n)$ respectively such that $s_i = t_i$ for all $2 \leq i \leq n$ and $\left\{s_1, t_1\right\} = \left\{\frac{1}{0}, k\right\}$, where $k \in \mathbb{Z}^+$.

From here, it is straightforward to observe that $\frac{a+c}{b+d}$ can be represented by the continued fraction $(r_1, r_2, ..., r_n)$, where $r_i = s_i = t_i$ for all $2 \leq i \leq n$ but $r_1 = k+1$. If $n$ is odd, these representations are consistent with the construction for rational tangles described in Section \ref{subsec: rational tangles}. Rational tangles $\frac{a}{b}, \frac{c}{d}$, and $\frac{a+c}{b+d}$ differ only at the first elementary rational tangles $[s_1], [t_1]$, and $[r_1]$. Resolving any crossing in $[r_1] = [k+1]$ results in the tangles $[k]$ and $\frac{1}{[0]}$, so $\frac{a}{b}$ and $\frac{c}{d}$ are the resolutions of $\frac{a+c}{b+d}$. 

If $n$ is even, we instead consider $\frac{b}{a}, \frac{d}{c}$, and $\frac{b+d}{a+c}$, which have continued fractions $(s_1, s_2, ..., s_n, 0)$, $(t_1, t_2, ..., t_n, 0)$, and $(r_1, r_2, ..., r_n, 0)$. However, from the odd $n$ case, $\frac{b}{a}$ and $\frac{d}{c}$ are the resolutions of $\frac{b+d}{a+c}$. The rational tangle $\frac{1}{r}$ can be obtained by rotating the rational tangle $r$ $90^{\circ}$ clockwise, reflecting across the vertical axis, and subsequently taking the mirror image. These transformations preserve an unoriented skein triangle, implying that $\frac{a}{b}$ and $\frac{c}{d}$ are the resolutions of $\frac{a+c}{b+d}$. This concludes the proof of the lemma.
\end{proof}

\begin{flushleft}
Since both a crossing and its crossing change result in the same resolutions, we can use Lemma \ref{lem: unoriented skein} to find an algebraic relation for an oriented skein triangle:
\end{flushleft}

\begin{lemma}\label{lem: oriented skein}
Suppose rational numbers $\frac{a}{b}$ and $\frac{c}{d}$ satisfy $|ad-bc|=1$. Then the unoriented tangles $\frac{a+c}{b+d}$ and $\frac{a-c}{b-d}$ differ by a crossing change. Furthermore, suppose they are given orientations compatible under this crossing change. Then exactly one of the oriented tangles $\frac{a}{b}$ and $\frac{c}{d}$ is the result of the oriented resolution of this crossing.
\end{lemma}

\begin{proof}
In an unoriented skein triangle $(T, T_0, T_1)$, if two of the members are fixed, there are exactly two possibilities for the third, which differ by a crossing change. Since $(\frac{a}{b}, \frac{c}{d}, \frac{a+c}{b+d})$ and $(\frac{a}{b}, \frac{c}{d}, \frac{a-c}{b-d})$ are unoriented skein triangles, it follows that $\frac{a+c}{b+d}$ and $\frac{a-c}{b-d}$ are related by a crossing change. After assigning these two tangles orientations compatible with this crossing change, the canonical resolution passes this orientation to exactly one of $\frac{a}{b}$ and $\frac{c}{d}$, as desired.
\end{proof}

\section{Unoriented Generating Set}
\subsection{Recursive Application of Unoriented Skein Triangles}
Recall that the set $\mathcal{S}$ of unoriented links is closed under the unoriented skein triangle if the following condition holds:
\begin{itemize}
    \item Suppose three links each with nonzero determinant constitute an unoriented skein triangle. Then if any two of them belong to the $\mathcal{S}$, so does the third. 
\end{itemize}
Our goal is to find a minimal generating set so that $\mathcal{S}$ is the set of all nonzero determinant unoriented links. Since all links are composed of a nexus of rational tangles (distinguished by the disks around them), we can focus on a particular disk. With a few initial generators, we can use Lemma \ref{lem: unoriented skein} recursively to replace the interior of the disk with other rational tangles, adding a variety of new links to $\mathcal{S}$.

\begin{lemma}\label{lem: tangle span}
Consider a set $\mathcal{T}$ containing rational tangles $\frac{0}{1}$ and $\frac{1}{0}$ subject to the condition: if any two members of an unoriented skein triangle of rational tangles belongs to $\mathcal{T}$, so does the third. Then $\mathcal{T}$ contains all unoriented rational tangles.
\end{lemma}

\begin{proof}
To show $\frac{p}{q} \in \mathcal{T}$ for all reduced fractions $\frac{p}{q} \in \mathbb{Q}^+$, we shall strong induct on $q$. The case $q = 0$ is already given, so we prove the hypothesis for $q = 1$. By Lemma \ref{lem: unoriented skein}, 
\[|ad - bc| = 1, \hspace{0.25cm} \frac{a}{b}, \hspace{0.125cm} \frac{c}{d} \in \mathcal{T} \Longrightarrow \frac{a+c}{b+d}, \hspace{0.125cm} \frac{a-c}{b-d} \in \mathcal{T}.\]
Then $\frac{1}{1}$ and $\frac{-1}{1}$ are contained in $\mathcal{T}$, and if $\frac{k}{1} \in \mathcal{T}$ for $k \geq 1$, $\frac{k+1}{1+0} = \frac{k+1}{1} \in \mathcal{T}$. Similarly, if $\frac{-k}{1} \in \mathcal{T}$ for $k \geq 1$, $\frac{-k-1}{1-0} \in \mathcal{T}$, so our set contains all rational tangles of the form $\frac{n}{1}$, where $n \in \mathbb{Z}$.

Now consider a reduced fraction $\frac{j}{k}$ and assume all reduced fractions of denominator $k-1$ or less are contained in $\mathcal{T}$. Let $q$ be the unique residue such that $qj \equiv -1 \pmod{k}$, which exists because gcd($j, k) = 1$. Setting $p = \frac{qj+1}{k}$, note that since gcd($p, q) = 1$ and $0 < q < k$, the inductive hypothesis implies $\frac{p}{q} \in \mathcal{T}$. Furthermore, we have the relation $kp - qj = 1$. Then defining $r = \frac{j(k-q)-1}{k}$ and $s = k-q$ ensures that gcd($r, s) = 1$ and $0 < s < k$, so $\frac{r}{s}$ is also contained in $\mathcal{T}$ by the inductive hypothesis. Fractions $\frac{p}{q}$ and $\frac{r}{s}$ satisfy $|sp - qr| = 1$ and $\frac{p+r}{q+s} = \frac{j}{k}$, so it follows that $\frac{j}{k} \in \mathcal{T}$. Thus, all reduced fractions with denominator $k$ are contained in $\mathcal{T}$, completing our induction. We conclude that $\mathcal{T}$ contains all unoriented rational tangles, as desired. 
\end{proof}

\begin{flushleft}
Thus we can generate all possible rational tangle replacements of a disk through applying a series of skein triangle conditions starting from just the trivial tangles. However, the challenge remains that some of these links may have zero determinant, which invalidates some of the skein triangles. The following lemmas are crucial in circumventing this issue.  
\end{flushleft}

\begin{lemma}\label{lem: determinant tangle lemma}
Given any link $L$, choose a disk containing a rational tangle and delete its contents, and let $L(\frac{p}{q})$ denote the link resulting from inserting the tangle $\frac{p}{q}$ into the disk. Then there exist integers $a$ and $b$ such that $\det(L(\frac{p}{q})) = |bp - aq|$ for all $\frac{p}{q} \in \mathbb{Q}^+$. 
\end{lemma}

\begin{proof} This can be deduced from Corollary $9.2$ of \cite{Lickorish}. For completeness, we provide a sketch of the argument. By Corollary $9.2$, the link determinant equals the order of the first homology of the branched double cover of a link $L$. Note that when passing to the branched double cover, putting a $\frac{p}{q}$ tangle corresponds to doing a surgery with coefficient $\frac{p}{q}$. Then a straightforward calculation verifies this result. 
\end{proof}

\begin{corollary}\label{lem: determinant of complexity one}
Without loss of generality, assume the orientation of the complexity $1$ link in figure 8. Then $\det(L(\frac{p}{q})) = |q|$. 
\end{corollary}

\begin{center}
\includegraphics[width = 5cm]{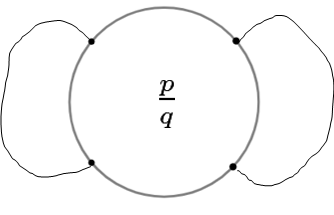}

Figure 8: Complexity $1$ link containing rational tangle $\frac{p}{q}$.
\end{center}

\begin{proof}
Inserting the rational tangle $\frac{1}{0}$ into the disk gives a split configuration with zero determinant, so by Lemma \ref{lem: determinant tangle lemma}, $|b\cdot 1 - a\cdot 0| = |b| = 0$; i.e. $b = 0$. Similarly, inserting $\frac{0}{1}$ gives the unknot, which has determinant $1$, so $|-a\cdot1| = |a| = 1$. It follows that for any rational tangle $\frac{p}{q}$, the determinant is $|bp-aq| = |q|$, as desired.
\end{proof}

\begin{lemma}\label{lem: zero determinant condition}
If there exists a rational tangle $\frac{r}{s}$ for which $L(\frac{r}{s})$ has nonzero determinant, then there is at most one $x \in \mathbb{Q}^+$ for which $L(x)$ has zero determinant.
\end{lemma}

\begin{proof}
Suppose on the contrary that $L(\frac{p}{q})$ and $L(\frac{p'}{q'})$ have zero determinant for distinct fractions $\frac{p}{q}$ and $\frac{p'}{q'}$. By Lemma \ref{lem: determinant tangle lemma}, this means that $bp = aq$ and $bp' = aq'$, where $a$ and $b$ are fixed. If $b$ is nonzero, $p = \frac{a}{b} q$ and $p' = \frac{a}{b} q'$, whence $pq' = p'q = \frac{a}{b}qq'$, a contradiction since $\frac{p}{q}$ and $\frac{p'}{q'}$ are distinct. Thus $b = 0$, but both $q$ and $q'$ cannot be $0$ otherwise both fractions are equal to $\frac{1}{0}$, so $a = 0$ as well. Thus, $L(x)$ has zero determinant for all $x \in \mathbb{Q}^+$, which is another contradiction if we set $x = \frac{r}{s}$. Then there is at most one $x$ for which $L(x)$ has zero determinant, as desired.
\end{proof}

\begin{flushleft}
The power of Lemma \ref{lem: zero determinant condition} is that it limits the number of possible invalid skein triangles. Since any given rational tangle belongs to infinitely many skein triangles, we can simply find other pathways to admit a certain link into $\mathcal{S}$.
\end{flushleft}

\subsection{Proof of the Unoriented Generating Set Theorem}
Our previous lemmas suggest that closure under the unoriented skein triangle can be used recursively to admit a large number of links into $\mathcal{S}$ from a limited number of generators. We now prove our first theorem:

\begin{theorem}
Given a set S of unoriented nonzero determinant links, suppose that it is closed under the
unoriented skein triangle and contains the unknot. Then it contains all nonzero determinant links.
\end{theorem}

\begin{proof}
We induct on the number of components of a link. For the base case, we show that all links with at most $2$ components and nonzero determinant are contained in $\mathcal{S}$. To prove this, we will induct on the complexity with the following base case:

\begin{lemma}\label{lem: complexity one in s}
All links of complexity $1$ and nonzero determinant are contained in $\mathcal{S}$.
\end{lemma}

\begin{proof}
Assume the orientation of figure $8$. Then any link obtained by placing the rational tangle $\frac{n}{1}$ in the disk for $n \in \mathbb{Z}$ is contained in $\mathcal{S}$, since it is isotopic to the unknot. By Lemma \ref{lem: tangle span}, this is enough to imply $L(\frac{p}{q}) \in \mathcal{S}$ for all $\frac{p}{q} \in \mathbb{Q}^+$ except $\frac{p}{q} = \frac{1}{0}$. All such links have determinant $|q| \neq 0$ from Corollary \ref{lem: determinant of complexity one}, so all links of complexity $1$ and nonzero determinant are contained in $\mathcal{S}$, as desired.
\end{proof}

Now assume that a nonzero determinant link $L$ has complexity $n$ with at most $2$ components, and that all such links with complexity strictly less than $n$ satisfy the hypothesis. We want to show that $L$ is contained in $\mathcal{S}$. Note there exists a disk in the complexity decomposition such that deleting the rational tangle inside yields $2$ strands. Indeed, if $L$ is a knot, any such disk will do; deleting the contents gives $4$ endpoints, and since a strand has precisely $2$ endpoints, there are exactly $2$ strands, $L_1$ and $L_2$. 

However, if $L$ has $2$ components, note that they intersect at some crossing, which in turn belongs to a rational tangle in a disk $C_1$. If the components do not intersect, the resulting split configuration has determinant zero by Lemma \ref{lem: split config}. $C_1$ intersects each component twice, so we can choose a rational tangle merging the two components, forming a knot. Then we can just delete the interior of $C_1$ again, giving us our two strands $L_1$ and $L_2$. Suppose $C_1$ originally contained the rational tangle $\frac{a}{b}$. We are left with the following two cases:

\begin{flushleft}
\textit{Case 1}: $L_1$ and $L_2$ do not intersect.
\end{flushleft}

By Lemma \ref{lem: zero determinant condition}, there is at most one rational tangle for which the determinant is $0$ upon insertion into $C_1$, since by assumption $L$ has nonzero determinant. Moreover, inserting the tangle $\frac{1}{0}$ forms a split configuration with zero determinant by Lemma \ref{lem: split config}, so inserting all other rational tangles in $C_1$ must result in a nonzero determinant link.

Placing tangles of the form $\frac{n}{1}$ for all $n \in \mathbb{Z}$ into $C_1$ have the resulting configuration of the connected sum of $L_1$ and $L_2$. Then we have the following lemma:

\begin{lemma}\label{lem: connected sum lemma}
$K_1 \in \mathcal{S}$ and $K_2 \in \mathcal{S}$ implies $K_1\#K_2 \in \mathcal{S}$.
\end{lemma}

\begin{proof}
A link $L$ is contained in $\mathcal{S}$ if there exists a sequence of nonzero determinant links $(L^0, L^1, ..., L^n)$ with $L^0, L^1 \in \mathcal{S}$ and $L^n = L$ such that for any $L^m$ with $2 \leq m \leq n$, there exist $L^i$ and $L^j$ satisfying $i, j < m$ which form an unoriented skein triangle with $L^m$. By assumption there exists such sequences to obtain $K_1$ and $K_2$. In addition, we can view $K_2$ as a connected sum between the unknot and itself. If $K_1$ has sequence $(K^0, K^1, ..., K^n)$, we can isotope this unknot obtain both $K^0$ and $K^1$, which must be isotopic to the unknot. Now consider the sequence $(K^0\#K_2, K^1\#K_2, ..., K^n\#K_2)$, where the first two terms are already in $\mathcal{S}$. Passing to the sequence $(K^0, K^1, ..., K^n)$, it follows that for any $K^m\#K_2$ with $2 \leq m \leq n$, there exist $K^i\#K_2$ and $K^j\#K_2$ satisfying $i, j < m$ that constitute an unoriented skein triangle with $K^m\#K_2$. By Proposition \ref{lem: connected sum multiply}, $\det(K^m\#K_2) = \det(K^m) \cdot \det(K_2)$, which is nonzero since $K^m$ and $K_2$ have nonzero determinant. Thus $K^n\#K_2 = K_1\#K_2$ is contained in $\mathcal{S}$, as desired. 
\end{proof}

Thus, all of the links given by inserting integer rational tangles are contained in $\mathcal{S}$, so Lemma \ref{lem: tangle span} implies $L \in \mathcal{S}$, as desired.

\begin{flushleft}
\textit{Case 2}: $L_1$ and $L_2$ do intersect. 
\end{flushleft}

Let disk $C_2$ in the complexity decomposition contain the intersection of $L_1$ and $L_2$, belonging to tangle $\frac{c}{d}$. Delete the contents of $C_2$, which means there are four components $l_1, l_2, l_3,$ and $l_4$, each of which connecting $C_1$ to $C_2$. Let $L(x, y)$ denote the link obtained by inserting rational tangles $x$ and $y$ into $C_1$ and $C_2$ respectively.

\begin{lemma}\label{lem: at most one}
For any rational tangle $x$ there exists at most one rational tangle $y$ such that $\det(L(x, y)) = 0$.
\end{lemma}

\begin{proof} Suppose there are more than one such rational tangles, and from Lemma \ref{lem: zero determinant condition}, this means the determinant is zero for all $y$. However, if $x$ without loss of generality connects components $l_1$ to $l_2$ and $l_3$ to $l_4$, we can choose a suitable rational tangle $z$ that connects $l_1$ to $l_3$ and $l_2$ to $l_4$, resulting in a link with one component; i.e. a knot. However, a knot always has odd determinant by Lemma \ref{lem: knot odd determinant}, so $\det(L(x, z)) \neq 0$, contradiction. Then only one such $y$ exists, as desired. \end{proof}

Thus, there is at most one $y$ for which $\det(L(\frac{0}{1}, y)) = 0$. For all other $y$, the inductive hypothesis implies that $L(\frac{0}{1}, y)$ is contained in $\mathcal{S}$ since its complexity is less than $n$ and it has nonzero determinant. Similarly, there is at most one $y$ for which $L(\frac{0}{1}, y) \notin \mathcal{S}$. We can extend this notion further:

\begin{lemma}\label{lem: finitely many}
For all rational tangles $x$, there are only finitely many rational tangles $y$ for which $L(x, y) \not\in \mathcal{S}$.
\end{lemma}

\begin{proof}
Consider a skein triangle $L(p, w), L(q, w),$ and $L(r, w)$ with $L(p, w)$ and $L(q, w)$ satisfying the hypothesis of the lemma. Let $\mathcal{U}$ be the set of all $y$ such that one of $L(p, y)$ or $L(q, y)$ is not in $\mathcal{S}$ or $\det(L(r, y)) = 0$, and note $\mathcal{U}$ is finite by our assumption and Lemma \ref{lem: at most one}. For all $z \not\in \mathcal{U}$, $L(p, z)$ and $L(q, z)$ are in $\mathcal{S}$ and $\det(L(r, z)) \neq 0$, so by closure under the unoriented skein triangle, $L(r, z) \in \mathcal{S}$; i.e., there are only finitely many $y$ for which $L(r, y) \not\in \mathcal{S}$. Starting with $L(\frac{0}{1}, y)$ and $L(\frac{1}{0}, y)$, it follows from Lemma \ref{lem: tangle span} that there are only finitely many $y$ for which $L(x, y) \not\in \mathcal{S}$ for all rational tangles $x$.
\end{proof}

In particular, take $r = \frac{a}{b}$ and $s = \frac{c}{d}$. By assumption $L(\frac{a}{b}, \frac{c}{d})$ has nonzero determinant. There exists $\frac{c'}{d'}$ satisfying $|cd' - dc'| = 1$, since we can set $c'$ equal to the unique residue $n \equiv -d^{-1} \pmod{c}$ and $d'$ equal to $\frac{c'd+1}{c}$. Consider $L(\frac{a}{b}, \frac{c' + mc}{d' + md})$ and $L(\frac{a}{b}, \frac{c' + (m+1)c}{d' + (m+1)d})$ and note that for sufficiently large $m$, these two links are contained in $\mathcal{S}$, as there are only finitely many $y$ for which $L(\frac{a}{b}, y) \not\in \mathcal{S}$ by Lemma \ref{lem: finitely many}. Then $L(\frac{a}{b}, \frac{c}{d}), L(\frac{a}{b}, \frac{c' + mc}{d' + md}),$ and $L(\frac{a}{b}, \frac{c' + (m+1)c}{d' + (m+1)d})$ form an unoriented skein triangle, and we can use the closure condition to conclude that $\mathcal{S}$ contains $L(\frac{a}{b}, \frac{c}{d}) = L$, as desired. The conclusions of Cases $1$ and $2$ show that for all nonzero determinant links $L$ with at most $2$ components, $L$ is contained in $\mathcal{S}$.  

Our base case is thus completed. Assume a link with nonzero determinant $L$ has $k$ components and that the hypothesis holds for all nonzero determinant $k-1$ component links. Again, we want to show $L \in \mathcal{S}$. Find two components, say $L_1$ and $L_2$, that intersect at a crossing. Such two components must exist, otherwise we have a split configuration and $\det(L) = 0$ by Lemma \ref{lem: split config}, contradiction. Consider the disk $D$ in the complexity decomposition containing this crossing, and suppose $\frac{p}{q}$ is the rational tangle in the interior. After deleting the contents of $D$, two of its endpoints belong to one component, while the other two endpoints belong to the other component. Equipped with Lemma \ref{lem: endpoint connection}, we can finally complete the proof of our theorem. 

Let $L(x)$ denote the link obtained after inserting rational tangle $x$ into $D$. There exists $\frac{p'}{q'}$ satisfying $|pq' - qp'| = 1$, so consider $L(\frac{p}{q}), L(\frac{p' + mp}{q' + mq}),$ and $L(\frac{p' + (m+1)p}{q' + (m+1)q})$, which form an unoriented skein triangle for all $m \in \mathbb{Z}$. Moreover, since $L(\frac{p}{q})$ has nonzero determinant, Lemma \ref{lem: zero determinant condition} states that there is at most one $x$ for which $\det(L(x)) = 0$, so take $m$ sufficiently large to avoid this $x$. Thus the links in the skein triangle all have determinant nonzero.

Because $|pq' - qp'| = 1$, $\frac{p}{q}$ is not congruent to $\frac{p'}{q'} \pmod 2$. Thus, out of $\frac{p}{q}, \frac{p' + mp}{q' + mq}$ and $\frac{p' + (m+1)p}{q' + (m+1)q}$, no two are congruent modulo $2$, so each corresponding rational tangle connects the endpoints of $D$ differently by Lemma \ref{lem: endpoint connection}. Inserting $\frac{p' + mp}{q' + mq}$ and $\frac{p' + (m+1)p}{q' + (m+1)q}$ will merge components $L_1$ and $L_2$, which implies $L(\frac{p' + mp}{q' + mq})$ and $L(\frac{p' + (m+1)p}{q' + (m+1)q})$ have $k-1$ components. By the inductive hypothesis, $L(\frac{p' + mp}{q' + mq})$ and $L(\frac{p' + (m+1)p}{q' + (m+1)q})$ are contained in $\mathcal{S}$, and using closure under the unoriented skein triangle, it follows that $L \in \mathcal{S}$.

This completes our induction, and we can conclude that all links with nonzero determinant are contained in $\mathcal{S}$, as desired.
\end{proof}

\begin{flushleft}
As a direct corollary, we can determine a nonzero determinant link invariant satisfying an unoriented skein relation by its value on the unknot:
\end{flushleft}

\begin{corollary}
Let $\chi$ be an invariant of non-zero determinant unoriented links and suppose $\chi$ has a recursive formula for the following triple of links differing at a single crossing:
\[(L, L_0, L_1) \hspace{0.25cm} \text{whenever} 
\begin{array}{cc}
    \det(L) \\
     \det(L_0) \\ 
      \det(L_1) \\
\end{array} \Bigg\} \neq 0
\]
Then $\chi$ is completely determined by its value on the unknot.
\end{corollary}

\section{Oriented Generating Set}
\begin{flushleft}
We shall prove our second theorem, which considers oriented links. $\mathcal{S}$ is now closed under the oriented skein triangle:
\begin{itemize}
\item Suppose three links with nonzero determinant constitute an oriented skein triangle. Then if any two of these belong to $\mathcal{S}$, so does the third.
\end{itemize}
\end{flushleft}

\begin{flushleft}
Although the implementation of Lemma \ref{lem: oriented skein} requires significantly more care than that of Lemma \ref{lem: unoriented skein}, many of our tools from the unoriented case translate directly to the oriented case. The oriented skein relation is weaker than the unoriented relation, so we use a slightly larger generating set for $\mathcal{S}$:
\end{flushleft}

\begin{theorem}
Given a set S of oriented nonzero determinant links, suppose that it is closed under the
oriented skein triangle and contains the unknot and all orientations of the Hopf link. Then it contains all
nonzero determinant links.
\end{theorem}

\begin{proof}
We will induct on the number of components, our base case asserting that all oriented links with at most two components and nonzero determinant are contained in $\mathcal{S}$.

For a planar projection of a nonzero determinant oriented link $L$ with at most two components (these components must intersect, otherwise the split configuration has zero determinant by Lemma \ref{lem: split config}), recall there exists a disk with a rational tangle inside that partitions $L$ into the interior of the disk and exactly two strands outside the disk: $L_1$ and $L_2$. Note that the respective endpoints of the two strands are either (based on figure $6$) $ab$ and $cd$, $ac$ and $bd$, or $ad$ and $bc$. 

Thus any oriented $L$ may be represented by choosing a disk that partitions the oriented link diagram into strands $L_1$, $L_2$, and the interior of the disk. An \textit{open configuration} is then obtained by deleting the contents of the disk.  Let $O(x)$ denote the link obtained by inserting rational tangle $x$ into the disk. Note $O$ has a preexisting orientation that is passed onto $x$. As a result, the insertion of $x$ may not produce an orientation compatible link. We claim that for any open configuration and rational tangle $x$, if the corresponding $O(x)$ is orientation compatible and has nonzero determinant, then $O(x) \in \mathcal{S}$.

To show this, we shall perform a double induction on two quantities: the \textit{lifting number} and the \textit{external crossing number}. If $\mathcal{C}$ is the set of crossings between $L_1$ and $L_2$, then the lifting number is the minimal number of crossing changes among $\mathcal{C}$ required to isotopically separate $L_1$ from $L_2$ (without moving end points) in the complement of the ball in $S^{3}$ that corresponds to the disk on the plane. The external crossing number is the number of crossings outside the disk.

Let $\mathcal{P}(x, y)$ denote the assertion that all open configurations with lifting number at most $x$ and external crossing number at most $y$ satisfy the hypothesis of the claim. Note that the lifting number is automatically $0$ if the external crossing number is zero. In this case, we get a complexity $1$ link, so $\mathcal{P}(k, 0)$ for $k \geq 0$ is true by the following lemma:

\begin{lemma}\label{lem: oriented complexity one}
All oriented nonzero determinant links of complexity $1$ are contained in $\mathcal{S}$.
\end{lemma}

\begin{proof}
Let $O$ be an open configuration such that $O(x)$ has complexity $1$. For this proof, instead of imposing the orientation of $O$ onto $x$, we shall consider the equivalent notion of inserting an oriented rational tangle $x$ and imposing its orientation onto $O$. In other words, we insert such an oriented $x$ into figure $8$, but $x$ must have an orientation compatible with this figure. Call an oriented rational tangle \textit{parallel} if the two strands have the same orientation and \textit{antiparallel} otherwise. Using Lemma \ref{lem: endpoint connection}, if $\frac{p}{q}$ is such that $q$ is even, any orientation is compatible since the top two endpoints belong to different components. If $q$ is odd, the resulting link has $1$ component, so $p$ even and odd allow for strictly antiparallel and parallel orientations respectively.

To show $O(\frac{p}{q}) \in \mathcal{S}$ for all orientation compatible rational tangles $\frac{p}{q}$, we shall strong induct on $p$ in the positive direction. The negative direction follows similarly. The base case $p = 0$ is the oriented unknot, so we now examine $p = 1$, again using induction. $O(\frac{1}{1})$ is in $\mathcal{S}$ in parallel orientation as it is isotopic to the unknot. Moreover, all orientations of the Hopf link are already in $\mathcal{S}$, so $O(\frac{1}{2}) \in \mathcal{S}$ for both the parallel and antiparallel orientations of $\frac{1}{2}$. Assuming $O(\frac{1}{k}) \in \mathcal{S}$ for all orientation compatible $\frac{1}{k}$ with $0 < k < n$, we now consider $O(\frac{1}{n})$. If $n$ is odd, $\frac{1}{n}$ must have parallel orientation, which implies $(O(\frac{1}{n-2}), O(\frac{1}{n-1}), O(\frac{1}{n}))$ is an oriented skein triangle. For $n \geq 3$, all three of these links have nonzero determinant by Corollary \ref{lem: determinant of complexity one}, admitting $O(\frac{1}{n}) \in \mathcal{S}$. 

Similarly, if $n$ is even, we can use the oriented skein triangles $(O(\frac{1}{n-2}), O(\frac{1}{n-1}),$ $O(\frac{1}{n}))$ or $(O(\frac{1}{n-2}), O(\frac{0}{1}), O(\frac{1}{n}))$ if $\frac{1}{n}$ has parallel or antiparallel orientation respectively. For $n > 2$, all such links have nonzero determinant, implying $O(\frac{1}{n}) \in \mathcal{S}$. Therefore $O(\frac{1}{m}) \in \mathcal{S}$ for all orientation compatible $\frac{1}{m}$ and $m > 0$, completing the base case $p = 1$.

We now proceed in a similar manner to Lemma \ref{lem: tangle span}. For fixed $k > 1$, suppose that the oriented $O(\frac{a}{b})$ is contained in $\mathcal{S}$ for all $0 \leq a < k$. We will show $O(\frac{k}{ki + j}) \in \mathcal{S}$ for all compatibly oriented $\frac{k}{ki + j}$ with $i \geq 0$ and $0 < j < k$ such that gcd$(j, k) = 1$. The case $(i, j) = (0, 1)$ is equivalent to the oriented unknot, so we may assume $(i, j) \neq (0, 1)$. Let $p$ be the unique residue congruent to the inverse of $j \pmod{k}$ and set $q = \frac{pki + pj - 1}{k}$, so that $q$ is a positive integer and satisfies gcd$(p, q) = 1$.

Similarly, let $r = k - p$ and set $s = \frac{rki + rj - 1}{k}$, which gives gcd$(r, s) = 1$. In fact, $\frac{p+r}{q+s} = \frac{k}{ki+j}$, $|ps - rq| = 1$, and $|p - r| < k$ because $p, r < k$. By the inductive hypothesis and Lemma \ref{lem: oriented skein}, one of $(O(\frac{k}{ki+j}), O(\frac{p}{q}), O(\frac{p-r}{q-s}))$ and $(O(\frac{k}{ki+j}), O(\frac{r}{s}), O(\frac{p-r}{q-s}))$ is an oriented skein triangle for each compatible orientation of $\frac{k}{ki+j}$. Therefore $O(\frac{k}{ki + j}) \in \mathcal{S}$ for all compatibly oriented $\frac{k}{ki + j}$, using the closure condition on $\mathcal{S}$. This concludes our induction, so all oriented nonzero determinant links of complexity $1$ are contained in $\mathcal{S}$, as desired.
\end{proof}

Now assume that we have established $\mathcal{P}(k,j-1)$ for some fixed $j \geq 1$ and all $k \geq 0$. We want to show that $\mathcal{P}(k, j)$ is true for all $k \geq 0$. Beginning with the $\mathcal{P}(0,j)$ case, consider any open configuration $O$ with lifting number $0$ and external crossing number $j$, and note that $L_1$ and $L_2$ can be separated. If both strands are trivial; i.e., all external crossings are contained in $\mathcal{C}$, $O(x)$ corresponds to a complexity $1$ link, so $\mathcal{P}(0, j)$ is true by Lemma \ref{lem: oriented complexity one}. 

However, if at least one of the strands is nontrivial, we can decompose $O(x)$ into a connected sum of two links which can be represented by open configurations with strictly less external crossing number. For $O(x)$ with nonzero determinant, Proposition \ref{lem: connected sum multiply} implies the two links must also have nonzero determinant. By the inductive hypothesis, they are contained in $\mathcal{S}$, so we may repeat the argument of Lemma \ref{lem: connected sum lemma} verbatim to show $O(x) \in \mathcal{S}$, establishing $\mathcal{P}(0, j)$.

Finally, we must show that if $\mathcal{P}(i, j-1)$ and $\mathcal{P}(i-1, j)$ both hold, then so does $\mathcal{P}(i, j)$ for fixed $i \geq 1$. Consider any open configuration $O$ with lifting number $i$ and external crossing number $j$. There exists a sequence of crossings for the lifting number in $O$ that will separate $L_1$ and $L_2$ after a crossing change is applied to each of them in sequence. Choose the first crossing in that sequence, and suppose its crossing change and resolution result in $O_1$ and $O_2$ respectively.

Clearly $O_1$ is still an open configuration, and since the crossings for the lifting number are between two different strands $L_1$ and $L_2$, $O_2$ gains no new components from the resolution and is also an open configuration. Hence, $(O(x), O_1(x), O_2(x))$ is an oriented skein triangle for any $x$. We will need the following oriented version of Lemma \ref{lem: zero determinant condition}:

\begin{lemma}\label{lem: oriented zero det cond}
For any open configuration, there exists at most one rational tangle that may be inserted into the disk so that the resulting link is orientation compatible and has zero determinant.
\end{lemma}

\begin{proof}
Lemma \ref{lem: endpoint connection} gives three types of rational tangles based on paired numerator-denominator parity that connect the endpoints of the disk in three different ways. Exactly two of these types will merge strands $L_1$ and $L_2$ so that the resulting link (neglecting orientation) is a knot. 

Furthermore, exactly two of the three types will be compatible with the orientation, as for each of $L_1$ and $L_2$, one endpoint must exit the disk and one endpoint must enter the disk since the orientation runs along the same strand. This means there are exactly two exiting and two entering endpoints. Since an exiting endpoint must be connected to an entering endpoint for orientation compatibility, there are exactly two ways for a rational tangle to connect the endpoints of the disk to produce an oriented link.

It follows that there must be at least one type of rational tangle whose insertion forms an orientation compatible link, which has nonzero determinant. By Lemma \ref{lem: zero determinant condition}, there is at most one rational tangle which can be placed inside the disk for the resulting link to have determinant zero, as desired.
\end{proof}

Returning to the $(O(x), O_1(x), O_2(x))$, Lemma \ref{lem: oriented zero det cond} states for each member of the oriented skein triangle, there is at most one rational tangle whose insertion obtains an orientation compatible link with zero determinant. Suppose $t$ is not one of these rational tangles and is compatible with the orientation of $O$ (and hence with the orientations of $O_1$ and $O_2$). By the inductive hypothesis $O_1(t)$ and $O_2(t)$ are contained in $\mathcal{S}$, which implies the same for $O(t)$. This means that for all but finitely many orientation compatible $x$, $O(x) \in \mathcal{S}$.

We now show that if $O(\frac{p}{q})$ has nonzero determinant and $\frac{p}{q}$ is orientation compatible with $O$, the resulting oriented link is contained in $\mathcal{S}$. Let $\frac{p'}{q'}$ be such that $|pq' - qp'| = 1$, and consider $\frac{p'+ (m+1)p}{q' + (m+1)q}$, which forms an unoriented skein triangle with $\frac{p}{q}$ and $\frac{p'+ mp}{q' + mq}$. These three fractions have distinct numerator-denominator parity types, and since exactly two such types are orientation compatible, we may choose the parity of $m$ so that $\frac{p'+ (m+1)p}{q' + (m+1)q}$ and $\frac{p}{q}$ are compatible. Hence $(O(\frac{p}{q})$, $O(\frac{p'+ (m+1)p}{q' + (m+1)q})$, $O(\frac{p'+ (m-1)p}{q' + (m-1)q}))$ is an oriented skein triangle. For all but finitely many orientation compatible $x$, $O(x) \in \mathcal{S}$, so by choosing a sufficiently large $m$ with the appropriate parity, $O(\frac{p'+ (m+1)p}{q' + (m+1)q})$ and $O(\frac{p'+ (m-1)p}{q' + (m-1)q})$ must be contained in $\mathcal{S}$. Since $\mathcal{S}$ is closed under the oriented skein triangle, $O(\frac{p}{q}) \in \mathcal{S}$. This proves the assertion $\mathcal{P}(i, j)$ is true, which completes our induction.

We conclude that for any open configuration and rational tangle $x$, if the corresponding $O(x)$ is orientation compatible and has nonzero determinant, then $O(x) \in \mathcal{S}$. Since any oriented link with nonzero determinant and at most two components may be expressed in this way, all such links are contained in $\mathcal{S}$. 

Now that the base case is proven, suppose $L$ has $k$ components and nonzero determinant and that all oriented links with $k-1$ components and nonzero determinant are already contained in $\mathcal{S}$. Choose an intersection between two components which belongs to a rational tangle $\frac{p}{q}$ in disk $D$. Again, such an intersection exists otherwise the resulting split configuration would have zero determinant by Lemma \ref{lem: split config}. Deleting the interior of $D$, by Lemma \ref{lem: zero determinant condition} there exists at most one orientation compatible $x$ for which $L(x)$ has zero determinant. Considering the oriented skein triangle $(L(\frac{p}{q}), L(\frac{p'+ (m+1)p}{q' + (m+1)q}), L(\frac{p'+ (m-1)p}{q' + (m-1)q}))$, by the arguments above we may choose $m$ with sufficient parity and size so that $L(\frac{p'+ (m+1)p}{q' + (m+1)q})$ and $L(\frac{p'+ (m-1)p}{q' + (m-1)q})$ have nonzero determinant. Because these links also have $k-1$ components, they are contained in $\mathcal{S}$ by the inductive hypothesis, whence $L(\frac{p}{q}) \in \mathcal{S}$, as desired. This completes our original induction on the number of components, and we conclude that $\mathcal{S}$ contains all oriented links with nonzero determinant.
\end{proof}

\begin{flushleft}
With this theorem, we can greatly simplify Mullins theorem with the following corollary:
\end{flushleft}
\begin{corollary}
Let $\chi$ be an invariant of non-zero determinant oriented links and suppose $\chi$ has a recursive formula for the following triple of links differing at a single crossing:
\[(L_+, L_-, L_o) \hspace{0.25cm} \text{whenever} 
\begin{array}{cc}
    \det(L_+) \\
     \det(L_-) \\ 
      \det(L_o) \\
\end{array} \Bigg\} \neq 0
\]
Then $\chi$ is completely determined by its values on the unknot and all orientations of the Hopf link.
\end{corollary}

\section{Conclusions}
We have determined that all unoriented links with nonzero determinant can be generated by the unknot in a set subject to closure under the unoriented skein triangle. Moreover, all nonzero determinant oriented links can be generated by the oriented unknot and all orientations of the Hopf link in a set subject to closure under the oriented skein triangle. These theorems are extremely useful in characterizing nonzero determinant link invariants. For example, the computations of the Casson-Walker invariant in \cite{Mullins} and the Fr{\o}yshov invariant in \cite{Froyshov} can be immensely simplified. In addition, our theorems can be used to prove that two nonzero determinant link invariants coincide by showing they share the same skein relation and values on the unknot and the oriented Hopf links. This provides a powerful tool to establish new relations between classical link invariants. 

\section{Acknowledgements}

I would like to express my utmost gratitude to Dr. Jianfeng Lin of MIT, for proposing this project and mentoring me. I am extremely thankful for his invaluable help and support and for the time he spent with me throughout the research process. I would also like to thank Dr. Tanya Khovanova and Svetlana Makarova, both from MIT, who proofread my paper and provided very helpful suggestions. I am extremely grateful for PRIMES-USA and everyone involved in this program for providing me with this wonderful research opportunity.

\end{document}